\documentclass[12pt]{amsart}
\usepackage{amsmath, amsfonts,amsthm,times,graphics}

 \makeatletter
\renewcommand*\subjclass[2][2010]{%
  \def\@subjclass{#2}%
  \@ifundefied{subjclassname@#1}{%
    \ClassWarning{\@classname}{Unknown edition (#1) of Mathematics
      Subject Classification; using '2010'.}%
  }{%
    \@xp\let\@xp\subjclassname\csname subjclassname@#1\endcsname
  }%
}

 \makeatother

\newtheorem{theorem}{Theorem}
\newtheorem{lemma}{Lemma}
\newtheorem{corollary}{Corollary}
\newtheorem{conjecture}{Conjecture}
\newtheorem{proposition}{Proposition}
\newtheorem{fact}{Fact}
\newtheorem{definition}{Definition}
\theoremstyle{definition}
\newtheorem{remark}{Remarks}

\begin{document}

\title{Variations of  Kurepa's left factorial hypothesis}

 \author{Romeo Me\v strovi\' c}

\address{Maritime Faculty, University of Montenegro, 
Dobrota 36, 85330 Kotor, Montenegro}
\email{romeo@ac.me}

\vspace{5mm}

\maketitle
 \begin{abstract}
Kurepa's hypothesis asserts that 
for each integer $n\ge 2$ the greatest common divisor of 
$!n:=\sum_{k=0}^{n-1}k!$ and $n!$ is $2$. 
Motivated by an equivalent formulation of this hypothesis involving 
derangement numbers, here we give 
a formulation of  Kurepa's hypothesis in terms of 
divisibility of any Kurepa's determinant $K_p$ of order $p-4$
 by  a prime $p\ge 7$. 
In the previous version of this article we have proposed
the strong Kurepa's hypothesis involving 
a general Kurepa's determinant $K_n$ with any integer $n\ge 7$.
We prove the ``even part" of this hypothesis which 
can be considered as a generalization of Kurepa's hypothesis.
However, by using a congruence for $K_n$ involving 
the  derangement number $S_{n-1}$ with an odd integer $n\ge 9$, 
we find that the integer $11563=31\times 373$ is a counterexample to 
the ``odd composite part'' of strong Kurepa's hypothesis.    
 
We also present some remarks, divisibility properties 
and computational results closely related to the questions 
on Kurepa's hypothesis involving  derangement numbers and Bell numbers.  
 \end{abstract}

\vspace{4mm}

{\it $2010$ Mathematics subject classification}: Primary 05A10; 
Secondary 11B65, 11B73, 11A05, 11A07. 

\vspace{2mm}


  \section{Remarks on Kurepa's hypothesis}
In 1971 Dj. Kurepa \cite{ku} 
introduced the {\it left factorial function} $!n$
which is defined as 
  $$
!0=0,\quad !n=\sum_{k=0}^{n-1}k!,\quad n\in \Bbb N.
  $$ 
$!n$ is the  Sloane's sequence A003422 in \cite{sl}.

For more details of the following conjecture and its reformulations
 see a overview of A. Ivi\'c and \v{Z}. Mijajlovi\'{c} \cite{im1}.

 \begin{conjecture}[{\bf Kurepa's left factorial hypothesis}]
For each positive integer $n\ge 2$ the greatest common divisor of 
$!n$ and $n!$ is $2$.
 \end{conjecture}
Kurepa's hypothesis and its equivalent formulation
 appear in R. Guy's classic book \cite{gu} as problem B44 which asserts that
   $$
!n\not\equiv 0\pmod{n}\quad {\rm for\,\, all}\,\, n>2.
 $$ 
Alternating sums of factorials $\sum_{k=1}^{n-1}(-1)^{k-1}k!$
are involved in  Problem B43 in \cite{gu} which was solved 
by M. \v{Z}ivkovi\'{c} \cite{z}.

Further, Kurepa's hypothesis was tested by computers for $n<1000000$
by Mijajlovi\'{c} and Gogi\'{c} in 1991 (see e.g., \cite{mi} and \cite{ko}). 
Kurepa's left factorial hypothesis 
(or in the sequel, written briefly {\it Kurepa's hypothesis}) is an unsolved 
problem since 1971 and there seems to be no significant progress in solving 
it. 
Notice that a published proof of Kurepa's hypothesis in 2004 by 
D. Barsky and B. Benzaghou \cite[Th\'{e}or\`{e}me 3, p. 13]{bb} contains some 
irreparable calculation errors in the proof of Theorem 3 of this article 
\cite{bb2}, and this proof is therefore withdrawn.

 However, there are several statements equivalent to  Kurepa's hypothesis
(see e.g., Kellner \cite[Conjecture 1.1 and Corollary 2.3]{ke}, 
Ivi\'c and Mijajlovi\'{c} \cite{im1}, 
Mijajlovi\'{c} \cite[Theorem 2.1]{mi1},  Petojevi\'{c} 
\cite{pe1} and \cite[Subsection 3.3]{pe}, 
Petojevi\'{c}, \v{Z}i\v{z}ovi\'{c} and S. Cveji\'{c} 
\cite[Theorems 1 and 2]{pzc}, 
\v{S}ami \cite{sa}, Stankovi\'{c} \cite{st}, \v{Z}ivkovi\'{c} \cite{z}). 
Moreover, there are numerous identities involving 
the left factorial function $!n$ and related generalizations 
(see Carlitz \cite{ca}, Milovanovi\'{c} \cite{mil}, 
Petojevi\'{c} and Milovanovi\'{c} \cite{pm}, 
  Slavi\'{c} \cite{sla}, Stankovi\'{c} \cite{st},  Stankovi\'{c} and 
\v{Z}i\v{z}ovi\'{c} \cite{stz}). Moreover, 
Kurepa's hypothesis is closely related to the Sloane's sequences 
A049782, A051396, A051397, A052169, A052201, A054516 and A056158 
\cite{sl}.

  \begin{remark}
{\rm It was proved in \cite[p. 149, Theorem 2.4]{ku} that 
Kurepa's hypothesis is equivalent to the assertion that 
$!p\not\equiv 0(\bmod{\,p})$ for all odd primes $p$. According to 
Mijajlovi\'{c} \cite{mi1}, Kurepa's hypothesis is equivalent to 
  $$
\sum_{k=0}^{p-1}\frac{(-1)^k}{k!}\not\equiv 0\pmod{p}\quad 
{\rm for\,\, each\,\, prime}\quad p\ge 3.\leqno(1)
  $$
Notice that  $S_n:=n!\sum_{k=0}^n(-1)^k/k!$ $(n=0,1,2,\ldots)$
is the {\it subfactorial function} whose values   
are the well known {\it derangement numbers}  which give the number 
of permutations of $n$ elements without any 
fixpoints (Sloane's sequence A000166 in \cite{sl}. 
Certainly $S_{p-1}$ 
can take  any of the $p$ possible values $(\bmod{\,p})$.
Assuming that $S_{p-1}$ takes these values  randomly, 
the ``probability'' that $S_{p-1}$ takes any particular value
(say 0) is $1/p$. From this 
and using a heuristic argument based on ``log log philosophy'' 
(see e.g., \cite{me}), we might argue that  the  number of primes $p$ 
in an interval $[x,y]$ such that 
$\min\{S_{p-1}(\bmod{\,p}),p-S_{p-1}(\bmod{\,p})\}\le d$ for a ``small'' 
nonnegative integer $d$ is expected to be
       $$
K_d(x,y):=(2d+1)\sum_{x\le p\le y}\frac{1}{p}\approx (2d+1)
\log\frac{\log y}{\log x}.\leqno(2)
       $$
(Here the second estimate is a classical asymptotic formula of Mertens 
\cite[p. 94]{fi}). In particular, for the interval $[x,y]=[23,2^{23}]$ with $d=9$ 
the above estimate implies that  $K(23,2^{23})\approx 30.8977$.
M. \v{Z}ivkovi\'{c} \cite[Table 1]{z} verified that 
$S_{p-1}:=\sum_{k=0}^{p-1}(-1)^k/k!\not\equiv 0(\bmod{\,p})$ for all 
odd primes $p<2^{23}$. On the other hand, it follows by (2) that 
the expected number of such odd primes less than $2^{23}$ is about 
$\log\frac{\log 2^{23}}{\log 3}=2.67493$; 
also, the expected number of such odd primes from the interval 
$[353,2^{23}]$ is about $\log\frac{\log 2^{23}}{\log 353}=0.999729$.
Accordingly to these two expected numbers of primes 
which would be the ``counterexamples to  Kurepa's hypothesis''
and in view of related \v{Z}ivkovi\'{c}'s computation  up to 
$2^{23}$ \cite{z}, it may be of interest to determine 
``the probability'' that $S_{p-1}\not\equiv 0(\bmod{\,p})$ 
for each odd prime $p$ such that $x\le p<2^{23}$ for 
a given fixed $x$. Assuming 
the fact that for every pair of different odd primes $p$ and $q$ 
the events $A$ - ``$p$ satisfies the congruence $S_{p-1}\equiv 0(\bmod{\,p})$''
and $B$ - ``$q$ satisfies the congruence $S_{q-1}\equiv 0(\bmod{\,q})$''
are independent, and in view of the previously mentioned  
heuristic probability  argument, we find that the probability $P(K[x,y])$ of 
the {\it Kurepa's event} $K[x,y]$ - ``the congruence $S_{p-1}\equiv 
0(\bmod{\,p})$ ($3\le x\le y$) is satisfied for none odd prime $p$ such that 
$x\le p\le y$'' is equal to 
  $$
P(K[x,y])=\prod_{x\le p\le y\atop{p\,\, \rm a\,\, prime}}
\left(1-\frac{1}{p}\right)\leqno(3)
  $$
(the above product ranges over all odd primes $p$ with 
$x\le p\le y$). 

In particular, using the fact that  
the greatest prime which is less than $2^{23}=8388608$
is the $564163$th prime $p_{564163}=8388593$, applying (3) in 
{\tt Mathematica} 8, we find that, $P(K[4,2^{23}])=0.105652$. 
This means that the probabilty that 
$S_{p-1}:=\sum_{k=0}^{p-1}(-1)^k/k!\not\equiv 0(\bmod{\,p})$ 
for all primes $p$ with $5\le p<2^{23}$ is equal to $0.105652$.
Of course, the value $0.105652$ is not sufficiently  small, and 
hence, for the verificitaion of the truth of Kurepa's hypothesis,
it would be useful further computations of $S_{p-1}$ 
modulo primes $p>2^{23}$. For example, $P(K[2^{23},5000000])=0.899309$
shows that the probability that $S_{p-1}\equiv 0(\bmod{\,p})$
for at least one  prime less than the $3001134$th prime  
$p_{3001134}=49999991$ is greater than $10\%$.

Furthermore, from the right part of Table 1 in \cite{z} we also see  that in 
the interval $[23,2^{23}]$ there are 27 primes  $p$ satisfying the condition 
$\min\{S_{p-1}(\bmod{\,p}),p-S_{p-1}(\bmod{\,p})\}\le 9$. 
Our computation in {\tt Mathematica} 8 shows that 
in the interval $[1000,100000]$ there are 118  primes $p$ satisfying the 
condition $\min\{S_{p-1}(\bmod{\,p}),p-S_{p-1}(\bmod{\,p})\}\le 99$. 
On the other hand, by the estimate (2) the expected number of such primes 
is $\approx 199\log\frac{\log 100000}{\log 1000}=101.654$.

Moreover, by using the mentioned  heuristic argument, 
we might argue that  the  number of primes $p$ 
in the interval $[2^{23},10^{19}]$ such that 
$S_{p-1}\equiv 0(\bmod{\,p})$ is expected to be
$\log\frac{\log 10^{19}}{\log 2^{23}}\approx 1.00949$.
In other words, under the validity of presented heuristic arguments 
we have the following fact.}
\end{remark}

\begin{fact} 
Under the validity of heuristic arguments presented in Remarks $1$,
it can be expected one prime less than $10^{19}$ 
which is ``a counterexample'' to Kurepa's hypothesis.
 \end{fact}

\section{A linear algebra formulation of Kurepa's hypothesis}

Motivated by a linear algebra formulation of Kurepa's hypothesis 
given in \cite{me2} we give the following definition.

  \begin{definition}
{\rm For any integer $n\ge 7$ the {\it Kurepa's determinant} 
$K_n$ is the determinant of order $n-4$ defined as}
  $$
K_n:=\left|\begin{array}{rrrrrrrrrcrr}
1 & 1& 1 & 1 & 1 & 1 & \dots   &1 & 1  & 1 & 1 &3 \\
3 & 1 & 1 & 1 & 1 & 1 &\dots    &1 & 1 & 1 & 1 &2 \\
1 & 4 & 1 & 1 & 1 & 1 &\dots  &1 & 1 & 1 & 1 & 2 \\
0 & 1 & 5 & 1 & 1 & 1 &\dots  &1 & 1 & 1 & 1 & 2\\
0 & 0 & 1 & 6 & 1 & 1 &\dots  &1 & 1 & 1 & 1 & 2\\
0 & 0 & 0 & 1 & 7 & 1 & \dots &1   & 1 & 1 & 1 & 2\\
0 & 0 & 0 & 0 & 1 & 8 & \dots &1  & 1 & 1 & 1 & 2\\
\vdots & \vdots &\vdots &\vdots&\vdots &\vdots &\vdots &\vdots &\vdots&\ddots 
&\vdots &\vdots\\
0 & 0 & 0 & 0 & 0 & 0 &\dots  & 0 &  1 &  n-4 & 1 & 2\\
0 & 0 & 0 & 0 & 0 & 0 & \dots & 0 & 0 & 0 & 1 & -4
\end{array}\right|\leqno(4)
  $$
\end{definition}
First few values of $K_n$ are as follows: 
$K_7=\left|\begin{array}{rrr} 1 &1&3\\ 3 & 1 & 2\\ 0 &1 &-4\end{array}
\right|=15$, $K_8=-47$, $K_9=197$, $K_{10}=-1029$,
$K_{11}=6439$, $K_{12}=-46927$, $K_{13}=390249$, $K_{14}=-3645737$,
$K_{15}=37792331$, $K_{16}=-430400211$ and $K_{17}=5341017373$.

Motivated by the reformulation of Kurepa's hypothesis given by (1) of 
Remarks 1, and using a linear algebra approach 
(working in the field $\Bbb F_{p}=\{0,1,\ldots,p-1\}$ modulo $p$), 
we can establish a reformulation of Kurepa's hypothesis given by the 
following result.

\begin{theorem}[\cite{me2}] 
Let $p$ be an odd prime. Then the following statements are 
equivalent.
 \begin{itemize} 
\item[(i)] Kurepa's hypothesis holds, i.e., 
for each positive integer $n\ge 2$ the greatest common divisor of $!n$
and $n!$ is $2$.
\item[(ii)]  For each prime $p\ge 7$ the Kurepa's determinant 
$K_p$ satisfies the condition $K_p\not\equiv 0(\bmod{\,p})$.
 \end{itemize}
\end{theorem}

\begin{remark} 
In order to evaluate the Kurepa's determinant $K_p$
modulo a prime $p\ge 11$, we apply  numerous elementary 
transformations,  and work simultaneously modulo $p$. 
Then we obtain the following result which  in view of  (1) of Remarks 1 
gives an indirect proof of Theorem 1 (\cite{me2}; also 
see the first congruence in the proof of Proposition 4 in Section 
4 with $n=p$). 
  \end{remark}
   \begin{proposition} 
If $p$ is a prime greater than $5$, then 
  $$
K_p\equiv \frac{1}{8}\sum_{k=0}^{p-1}\frac{(-1)^k}{k!}.\leqno(5)
 $$
\end{proposition}
Finally, we propose the following conjecture which in view of Theorem 1 may 
be considered as the strong  Kurepa's  hypothesis.
 \begin{conjecture}[{\bf The strong Kurepa's hypothesis}]
For each integer $n\ge 7$ the Kurepa's determinant $K_n$ is not divisible 
by $n$.
 \end{conjecture}
If in the expression (4) for $K_n$ we replace 
every odd element by 1 and every even element by 0,
we obtain the following definition. 
  \begin{definition}
{\rm For any integer $n\ge 7$ the} {\it Kurepa's binary determinant} 
$K_n'$ {\rm is the determinant of order $n-4$ defined as}
  $$
K_n':=\left|\begin{array}{rrrrrrrrrccc}
1 & 1& 1 & 1 & 1 & 1 & \dots   &1 & 1  & 1 & 1 &1 \\
1 & 1 & 1 & 1 & 1 & 1 &\dots    &1 & 1 & 1 & 1 &0 \\
1 & 0 & 1 & 1 & 1 & 1 &\dots  &1 & 1 & 1 & 1 & 0 \\
0 & 1 & 1 & 1 & 1 & 1 &\dots  &1 & 1 & 1 & 1 & 0\\
0 & 0 & 1 & 0 & 1 & 1 &\dots  &1 & 1 & 1 & 1 & 0\\
0 & 0 & 0 & 1 & 1 & 1 & \dots &1   & 1 & 1 & 1 & 0\\
0 & 0 & 0 & 0 & 1 & 0 & \dots &1  & 1 & 1 & 1 & 0\\
\vdots & \vdots &\vdots &\vdots&\vdots &\vdots &\vdots &\vdots &\vdots&\ddots 
&\vdots &\vdots\\
0 & 0 & 0 & 0 & 0 & 0 &\dots  & 0 &  1 & (1-(-1)^n)/2  & 1 & 0\\
0 & 0 & 0 & 0 & 0 & 0 & \dots & 0 & 0 & 0 & 1 & 0
\end{array}\right|.\leqno(6)
  $$
$($Here $(1-(-1)^n)/2)=1$ if $n$ is odd, and 
$(1-(-1)^n)/2)=0$ otherwise$)$.
   \end{definition}
 \begin{remark}
A computation gives the following few values 
of Kurepa's binary determinant $K_n'$: 
$K_7'=1,K_8'=1,K_9'=-1,K_{10}'=-1,K_{11}'=1,
K_{12}'=1,K_{13}'=-1,K_{14}'=-1,K_{15}'=1,K_{16}'=1,K_{17}'=-1,
K_{18}'=-1$, which suggests the following result.
   \end{remark}
 \begin{proposition}
$K_{2n}'=K_{2n-1}'=(-1)^n$ for all $n\ge 4$.
   \end{proposition}

As an immediate consequence of Proposition 2 whose proof 
is given in Section 6, we get the following 
result.
 \begin{corollary}
For each integer $n\ge 7$  the Kurepa's determinant $K_n$
is an odd integer.
 \end{corollary}
\begin{proof}[Proof of Corollary $1$]
 Using the obvious fact that 
$K_n\equiv K_n'(\bmod{\,2})$ for all $n\ge 7$,
by Proposition 2 we have $K_n\equiv 1(\bmod{\,2})$ for all $n\ge 15$.
This together with the fact that $K_n$ is odd for $7\le n\le 14$
yields the assertion.  
  \end{proof}
Obviously, Corollary 1 implies the truth of Conjecture 2 for 
all even integers $n\ge 8$, that is, we have the following statement.
   \begin{theorem}
The strong Kurepa's hypothesis holds for each even integer 
$n\ge 8$.
  \end{theorem}

On the other hand, in Section 4 we show that 
``the odd composite part'' of strong Kurepa's hypothesis is not true,
that is, we prove the following result.
    \begin{theorem}
For $n=11563=31\times 373$ we have $K_{11563}\equiv 0(\bmod{\,11563})$.
Therefore, the strong Kurepa's hypothesis does not 
hold for each odd composite integer $n\ge 9$.
  \end{theorem}

Hence, our results concerning the strong Kurepa's hypothesis may be 
summarized  as follows.
 \vspace{2mm}

{\it The strong Kurepa's 
hypothesis can be  divided into the following three parts.

$\bullet$ The ``prime'' part which asserts that 
$K_p\not\equiv 0(\bmod{\,p})$   for each prime $p>5$. This part is  by 
Proposition $1$ and Theorem $1$ equivalent to Kurepa's hypothesis
$($Conjecture $1)$.

$\bullet$ The ``even part'' which asserts that 
$K_n\not\equiv 0(\bmod{\,n})$   for each even integer $n\ge 8$. 
This part is true by Theorem $2$.

$\bullet$ The ``odd composite part'' which asserts that 
$K_n\not\equiv 0(\bmod{\,n})$  for each odd composite integer $n\ge 9$. 
This part is disproved  by Theorem $3$}.

\section{Kurepa hypothesis and derangement numbers}

Let us consider the derangement numbers $S_n$ $(n=0,1,2,\ldots)$ 
defined as 
    $$
S_n=n!\sum_{k=0}^n\frac{(-1)^k}{k!}.\leqno(7)
   $$ 
The following result is itself interesting.
  \begin{proposition}
Let $n\ge 4$ be a composite positive integer, and let $d\ge 2$ be any 
proper divisor of $n$ with $n=ad$. Then 
  $$
S_{n-1}\equiv (-1)^{n+d}S_{d-1}\pmod{d}.\leqno(8)
  $$
 \end{proposition}
\begin{proof}
Take $n=ad$ with a positive integer $a$. Notice that by (7) $S_{n-1}$
can be written as
  $$
S_{n-1}=\sum_{k=0}^{ad-1}(-1)^k(k+1)(k+2)\cdots (ad-1).\leqno(9)
  $$
Notice that the set $D_k=\{k+1,k+2,\ldots, ad-2,ad-1\}$ contains 
an integer which is divisible by $d$ whenever $k+1\le (a-1)d$. 
Using this fact from (9) we find that
 \begin{eqnarray*}
S_{n-1} &\equiv& \sum_{k=(a-1)d}^{ad-1}(-1)^k(k+1)(k+2)\cdots (ad-1)
\pmod{d}\\
&\equiv& \sum_{j=0}^{d-1}(-1)^{j+(a-1)d}(j+1)(j+2)\cdots (d-1)\pmod{d}\\
&=& (-1)^{(a-1)d}\sum_{j=0}^{d-1}(-1)^{j}(j+1)(j+2)\cdots (d-1)\\
&=&(-1)^{n-d}S_{d-1}=(-1)^{n+d}S_{d-1},
 \end{eqnarray*}
as desired.
  \end{proof}

We are now ready to extend Theorem 2.1 of \cite{im1} and our Theorem 1
as follows.
 \begin{theorem} 
The following statements are equivalent:
 \begin{itemize} 
\item[(i)] Kurepa's hypothesis holds.
\item[(ii)]   
$S_{p-1}\not\equiv 0(\bmod{\,p})$ for each prime $p\ge 3$.
\item[(iii)]   
$S_{n-1}\not\equiv 0(\bmod{\,n})$ for each integer $n\ge 3$.
\item[(iv)] 
For each integer $n\ge 3$ the numerator of the fraction 
  $$
\sum_{k=0}^{n-1}\frac{(-1)^k}{k!}
 $$
written in reduced form is not divisible by $n$.
 \end{itemize}
\end{theorem}
 \begin{proof}
As noticed above, the equivalence $(i)\Leftrightarrow (ii)$ was 
attributed by Mijajlovi\'{c} \cite{mi1}. 

The equivalence $(ii)\Leftrightarrow (iii)$ obviously follows from 
Proposition 3.

To complete the proof it is suffices to show the implications 
$(iii)\Rightarrow (iv)$ and $(iv)\Rightarrow (ii)$. First suppose that
$(iii)$ is satisfied. 
For any fixed $n\ge 3$ let $a$ and $b$ be relatively prime positive 
integers such that 
  $$
\sum_{k=0}^{n-1}\frac{(-1)^k}{k!}=\frac{a}{b}.
   $$
Then notice that $c=(n-1)!/b$ is an integer and 
 $S_{n-1}=(n-1)!a/b=ac$. From this and the fact that  
by  $(iii)$, $S_{n-1}\not\equiv 0(\bmod{\,n})$,
it follows that $ac$ is not divisible by $n$. Therefore,
$a$ is not also  divisible by $n$, which yields the  
assertion $(iv)$. 

Finally, if $(iv)$ is satisfied, then for any prime  $p\ge 3$ set
 $$
\sum_{k=0}^{p-1}\frac{(-1)^k}{k!}=\frac{a}{b},
   $$
where $a$ and $b$ are relatively prime positive integers such 
that $a\not\equiv 0 (\bmod{\, p})$ and $b\not\equiv 0 (\bmod{\, p})$. 
Using this, by Wilson theorem we have 
  $$
S_{p-1}=(p-1)!\sum_{k=0}^{p-1}\frac{(-1)^k}{k!}=(p-1)!\frac{a}{b}
\equiv -\frac{a}{b}\pmod{p} \not\equiv 0\pmod{p}.
   $$
This yields the assertion $(ii)$ and the proof is completed.
  \end{proof}
  \begin{corollary}
Let $q_1,q_2,\ldots,q_l$ be odd distinct primes, let $e_1,e_2,\ldots,e_l$ 
be positive integers and let $r$ be a nonnegative integer  
such that $S_{q_i^{e_i}-1}\equiv r(\bmod{\,q_i^{e_i}})$
for all $i=1,2,\ldots,l$. Then for 
$n=q_1^{e_1}q_2^{e_2}\cdots q_l^{e_l}$ there holds 
    $$
S_{n-1}\equiv r\pmod{n}.\leqno(10)
    $$
In particular, $S_{n-1}\equiv 0(\bmod{\,n})$ if and only if 
$S_{q_i^{e_i}-1}\equiv 0(\bmod{\,q_i^{e_i}})$ for all $i=1,2,\ldots,l$.
 \end{corollary}

\begin{proof} The congruence (10) immediately follows from the 
fact that by the congruence (8) of Proposition 3 for all 
$i=1,2,\ldots,l$ we have
     $$
S_{n-1}\equiv S_{q_i^{e_i}-1}\pmod{q_i^{e_i}}\equiv  r\pmod{q_i^{e_i}}.
    $$
\end{proof}

Similarly, Proposition 3 yields the following result.

\begin{corollary} Let $n$ be an even positive integer with the prime 
factorization $n=2^eq_1^{e_1}q_2^{e_2}\cdots q_l^{e_l}$. 
If $S_{2^{e}-1}\equiv r(\bmod{\,2^{e}})$ and  
$S_{q_i^{e_i}-1}\equiv -r(\bmod{\,q_i^{e_i}})$ for all $i=1,2,\ldots,l$,
then 
$$
S_{n-1}\equiv r\pmod{n}.
 $$
  \end{corollary}

Notice that Table 1 of \cite{z} contains all primes 
$p<2^{23}=8388608$ such that $\min\{r_p,p-r_p\}\le 10$, where 
$r_p:=!p(\bmod{\,p})$. 
Using {\tt Mathematica 8} we obtain only the following four 
prime powers $p^e$ with $e\ge 2$ less than 100000 such that 
$\min\{r_{p^e},p^e-r_{p^e}\}\le 2$: $\{2^2,2^3,3^2,7^2\}$.
Using this, the set 
  $$
\{2,3,5,7,11,23,31,67,227,373,10331\}
  $$ 
of all primes $p$  less than  $100000$ of Table 1 in \cite{z} 
for which $\min\{r_p,p-r_p\}\le 2$, Corollaries 2 and 3, we immediately
obtain Table 1.

\vfill\eject

{{\bf Table 1.}} 
Integers $n$ with  $2\le n<100000$ for which $S_{n-1}\equiv r_n\,(\bmod\,n)$
with $r_n\in\{-2,-1,0,1,2\}$ and/or 
with related values $|r_n/n|\le 10^{-3}$  
 \begin{center}
{\small
  \begin{tabular}{cccccccc}    
$n$ & factorization of $n$ & $r_n$ & $|r_n/n|\cdot 10^3$ &
$n$ & factorization of $n$ & $r_n$ & $|r_n/n|\cdot 10^3$ \\\hline
 2 & 2 & 0    & $0$ & 681  & $3\times 227$ & -2  & $>1$  \\
 3  & 3 & 1    & $>1$ &  746 & $2\times 373$ & -2    & $>1$   \\
 4 & $2^2$ & 2  & $>1$ & 804  & $2^2\times 3\times 67$ & 2  & $>1$ \\
 5 & 5  & -1  & $>1$ &908  & $2^2\times 227$ & 2    & $>1$ \\
 6 & $2\times 3$ & 2    & $>1$ &1362  & $2\times 3\times 227$ & 2    & $>1$  \\
 7 & 7 & -1   & $>1$ &1492 & $2^2\times 373$ & -2    & $>1$  \\
 8 & $2^3$ & -2   & $>1$ &1541 & $23\times 67$ & -2    & $>1$  \\
 9 & $3^2$ & 1    & $>1$ & 2724 & $2^2\times 3\times 227$ & 2 & $0.734214$ \\
 11 & 11 & 1    & $>1$ & 2984 & $2^3\times 373$ & -2    & $0.670241$ \\
 12 & $2^2\times 3$ & 2   & $>1$ &3082 & $2\times 23\cdot 67$ & 2    
& $0.648929$\\
 23 & 23 & -2   & $>1$ &4623 & $3\times 23\times 67$ & -2 & $0.432619$  \\
 31 & 31 & 2     & $>1$ &5221 & $23\times 227$ & -2    & $0.383068$  \\
33 & $3\times 11$ & 1    & $>1$ &  6164 & $2^2\times 23\times 67$ & 2 
&$0.324464$\\
35 & $5\times 7$ & -1    & $>1$ & 9246 & $2\times 3\times 23\times 67$ & 2 &
 $0.216309$\\
 46 & $2\times 23$ & 2   & $>1$ & 10331 & 10331 & -2 &
 $0.193592$  \\
 49 & $7^2$ & -1   & $>1$ & 10442 & $2\times 23\times 227$ & 2 &
 $0.191534$ \\
 62 & $2\times 31$  &  -2   & $>1$ &  11563 & $31\times 373$ & 2 &
 $0.172965$ \\
  67 & 67  & -2    & $>1$ & 15209 & $67\times 227$ & -2 &
 $0.131501$ \\
 69  & $3\times 23$ & -2    & $>1$ & 15663 & $3\times 23\times 227$ & -2 &
 $0.127689$ \\
 92 & $2^2\times 23$ & 2    & $>1$ & 18492 & $2^2\times 3\times 23\times 67$ 
& 2 &  $0.108154$ \\
 99  & $3^2\times 11$ & 1    & $>1$ & 20662 & $2\times 10331$ & 2 &
 $0.096796$\\
 124  & $2^2\times 31$ & -2    & $>1$ & 20884 & 
$2^2\times 23\times 227$ & 2 &  $0.095767$ \\
 134  & $2\times 67$ & 2    & $>1$ & 23126 & $2\times 31\times 373$ & 2 &
 $0.086482$ \\
 138  & $2\times 3\times 23$ & 2    & $>1$ & 30418 & 
$2\times 67\times 227$ & 2 &  $0.065750$ \\
 201  & $3\times 67$ & -2    & $>1$ & 30993 & $3\times 10331$ & -2 & 
$0.064530$ \\
 227  & 227 & -2    & $>1$ & 31326 & $2\times 3\times 23\times 227$ & 2 &
 $0.063844$ \\
 245  & $5\times 7^2$ & -1 & $>1$ & 41324 & $2^2\times 10331$ & 2 & 
$0.048398$ \\
248  & $2^3\times 31$ & -2    & $>1$ & 45627 & $3\times 67\times 227$ & -2
& $0.043833$ \\
268  & $2^2\times 67$ & 2    & $>1$ & 46252 & $2^2\times 31\times 373$ & -2
& $0.043241$ \\
276  & $2^2\times 3\times 23$ & 2 & $>1$ & 60836 & $2^2\times 67\times 227$
& 2 &  $0.032875$ \\
373  & 373 & 2    & $>1$ & 61986  & $2\times 3\times 10331$ & 2 &  $0.032265$\\
402  & $2\times 3\times 67$ & 2    & $>1$ & 62652  & 
$2^2\times 3\times 23\times 227$ & 2 & $0.031922$ \\
454  & $2\times 227$ & 2  & $>1$ & 91254 & $2\times 3\times 67\times 227$  
& 2 &  $0.021916$  \\
 &  &  & &  92504 & $2^3\times 31\times 373$ &-2   & $0.021620$    
 \end{tabular}}
 \end{center}

 \begin{remark}
If $n\ge 9$  is  an odd composite integer, then by (8) for 
any   divisor $d\ge 3$ of $n$ we have
   $$
S_{n-1}\equiv S_{d-1}\pmod{d}.
  $$
\end{remark}
\section{The odd composite part of strong Kurepa's hypothesis is not true}

The odd part of  strong Kurepa's hypothesis (Conjecture 2) 
asserts that $K_n\not\equiv 0(\bmod{\,n})$ for each 
odd composite integer $n\ge 9$.
The residues  $s_n:=-8K_n(\bmod{\,n})$ with $|s_n|\le 10$ for 
$n<2500$, including the corresponding residues $r_n:=S_{n-1}(\bmod{\,n})$
are presented in Table 2 (cf. Table 3).\\

{{\bf Table 2.}} The odd integers $n$ with  $7\le n<2500$ for which 
$-8K_n\equiv s_n\,(\bmod\,n)$ with $s_n\in\{-10,-9,\ldots,-1,0,1,\ldots, 9,10\}$
and related values $r_n:=S_{n-1}(\bmod\,n)$

 \begin{center}
{\small
  \begin{tabular}{cccc}    
$n$ & factorization of $n$ & $s_n$ & $r_n$ \\\hline
  7 & 7 & -1 &  $-1$  \\
  9 & $3^2$ & -1  &  1  \\
 11 &  11 & 1  &    1 \\
  15 & $3\times 5$ & 2   & 4  \\
 21 & $3\times 7$ & -10   & -8  \\
23  & 23 & -2  & -2  \\
27  & $3^3$ & 8  & 10  \\
31  & 31 & 2  & -2  \\
33  & $3\times 11$ & -1   & 1  \\
35  & $3\times 5$ & -3   & -1  \\
39  & $3\times 13$ & 8  & 10  \\
49  & $7^2 $ &-3  & -1  \\
63  & $3^2\times 7$ & -10  & -8  \\
67  &  67 & -2  & -2  \\
69  & $3\times 23$ & -4  & -2  \\
95  & $5\times 19$ & 7  & 9  \\
99  & $3^2\times 11$ & -1  & 1  \\
117  & $3^2\times 13$ & 8  & 10  \\
121  & $11^2$ & 10   & 12  \\
123  & $3\times 41$ & 2   & 4  \\
201  & $3\times 67$ & -4   & -2  \\
205  & $5\times 41$ & 2   &  4 \\
227 & 227 & -2   &  -2  \\
245  & $5\times 7^2$ & -3   & -1  \\
351  & $3^3\times 13$ & 8   & 10  \\
373 & 373  & 2   & 2  \\
417 & $3\times 139$ & -7   & -5  \\
453 & $3\times 151$ & 8    & 10  \\
489 & $3\times 163$ & 2   & 4  \\
615 & $3\times 5\times 41$ & 2   & 4  \\
681 & $3\times 227$ & -4   & -2  \\
815 & $5\times 163$ & 2   &  4 \\
831 & $3\times 277$ & 5   &  7 \\
923 & $13\times 71$ & -5    & -3  \\
985 & $5\times 197$ & 7   &  9 \\
1541 & $23\times 67$ & -4    & -2  \\
1745  & $5\times 349$ & -8    & -6  
\end{tabular}}
 \end{center}

\vfill\eject

{\bf Table 3.} 
The values $K_n$, $S_{n-1}$ and $(8K_n+S_{n-1})(\bmod\,n)$ for $7\le n\le 21$

 \begin{center}
{\small
  \begin{tabular}{cccc}    
$n$ & $K_n$  & $S_{n-1}$ & $(8K_n+S_{n-1})(\bmod{\,n})$ \\\hline
  7 & 15  & 265  &  0  \\
  8 & -47 &  1854  & -2   \\
  9 &  197 & 14833  & 2   \\
  10 & -1029 & 133496  & 4    \\
  11 & 6439  &  1334961  & 0    \\
  12 & -46927 & 14684570  & 6   \\
  13 & 390249 & 176214841  & 0   \\
  14 & -3645737 & 2290792932  & 4    \\
  15 & 37792331 & 32071101049  & 2   \\
  16 & -430400211 & 481066515734   & -2    \\
  17 & 5341017373 & 7697064251745  & 0   \\
  18 & -71724018781  & 130850092279664  & 0   \\
  19 & 1036207207363983  &    2355301661033953 & 0   \\
  20 & -16024176975479  &    44750731559645106 & -6   \\
  21 & 264083895859409  &   895014631192902121 &  2  
 \end{tabular}}
  \end{center}
 \vspace{2mm}

Table 2 suggests the following congruence.
 \begin{proposition}
For each odd composite integer $n\ge 9$ there holds
  $$
8K_n\equiv -S_{n-1} +2\pmod{n}.
  $$ 
\end{proposition}
  \begin{proof}
It is proved in \cite{me2} that for each odd integer $n\ge 7$
   $$
K_n\equiv -3S_{n-5}-1+(n-7)!\cdot 180\pmod{n}.
  $$
If $n\ge 9$ is an odd composite integer, then it is easy to see 
that $(n-7)!\cdot 180\equiv 0(\bmod{\,n})$, which substituting 
into above congruence yields 
     $$
K_n\equiv -3S_{n-5}-1\pmod{n},
  $$
or multiplying by 8,
       $$
8K_n\equiv -24S_{n-5}-8\pmod{n}.
  $$
From the recurrence relation $S_{m}=mS_{m-1}+(-1)^m$ with $m=n-1$
we obtain $S_{n-1}=(n-1)S_{n-2}+1\equiv -S_{n-2}+1(\bmod{\,n})$.
Iterating this three times, we find that 
$S_{n-1}\equiv 24S_{n-5}+10(\bmod{\,n})$, which substituting 
in the above congruence gives 
        $$
8K_n\equiv -S_{n-1}+2\pmod{n},
  $$
as desired.
 \end{proof}
 \begin{proof}[Proof of Theorem $3$]
As an immediate consequence of Proposition 4, we immediately obtain 
that the  ``odd composite part'' of  strong Kurepa's hypothesis  is 
equivalent to 
  $$
S_{n-1}\not\equiv 2 \pmod{n}\quad  {\rm for\,\, each \,\,
odd\,\, composite\,\, integer\,\,} n\ge 9.
  $$ 
However, from Table 1 we see that $11563=11\times 373$
satisfies the congruence 
 $S_{11562}\equiv 2\pmod{11563}$ which 
by the congruence of Proposition 4 implies that 
$K_{11563}\equiv 0\pmod{11563}$, as asserted. 
  \end{proof}

 \begin{remark}
If $n\ge 9$ is an odd  composite integer with the prime 
factorization $n=q_1^{e_1}q_2^{e_2}\cdots q_l^{e_l}$
then by Corollary $2$, $S_{n-1}\equiv 2(\bmod{\,n})$
if and only if $S_{q_i^{e_i}-1}\equiv 2(\bmod{\,q_i^{e_i}})$
for each $i=1,2,\ldots ,l$. By \cite[Table 1]{z}
we know  that  $31$ and $373$ are the only primes less than 
$2^{23}=8388608$ satisfying the congruence 
$S_{p-1}\equiv 2(\bmod{\,p})$. Moreover, 
$S_{31^2-1}\equiv 467\not= 2(\bmod{\,31^2})$
and $S_{373^2-1}\equiv 2613\not= 2(\bmod{\,373^2})$.
These facts and Proposition 3 show that 
$n=11563=11\times 373$ is the only odd composite positive 
integer (i.e., a counterexample to the odd part of strong 
Kurepa's hypothesis) less than $2^{23}=8388608$ for which  
$S_{n-1}\equiv 2(\bmod{\,n})$. From Table 1 we also see that 
there exist 31 even positive integers less than 100000 
satisfying the congruence $S_{n-1}\equiv 2(\bmod{\,n})$.
  \end{remark}

 \section{Kurepa's hypothesis and Bell numbers}

Recall that the derangement numbers $S_n$ considered in the 
previous section  are closely related 
to the {\it Bell numbers} $B_n$ given by the recurrence   
  $$
B_{n+1}=\sum_{k=0}^n{n\choose k}B_k,\quad n=0,1,2,\ldots,
  $$
with $B_0=1$ (see e.g., \cite[p. 373]{gkp}).
$B_n$ gives  the number  of partitions of a set of cardinality $n$.
This is Sloane's sequence A000110 in \cite{sl}
whose terms $B_0,B_1,\ldots,B_8$ are as follows: 
$1,1,2,5,15,52,203,877,4140$. 

It is known (see e.g., \cite[Corollary 1.3]{sz}) that for any prime $p$ 
we have
  $$
B_{p-1}-1\equiv S_{p-1}\pmod{p}.\leqno(11)
  $$
The congruence (11) and the equivalence $(i) \Leftrightarrow (ii)$ of 
Theorem 3 yields that Kurepa's hypothesis is also 
equivalent with the statement that 
$$
B_{p-1}\not\equiv 1\pmod{p}\quad{\rm for\,\, each\,\, prime}\,\, p\ge 3.
 $$ 
 The idea of the proof of Kurepa's hypothesis
given by D. Barsky and B. Benzaghou \cite[Th\'{e}or\`{e}me 3, p. 13]{bb} 
(for a related discussion see B. Sury \cite[Section 4]{su}) 
is to consider what is known as the Artin-Schreier extension 
$\Bbb F_p[\theta]$ of the field $\Bbb F_p$ of $p$ elements,
where $\theta$ is a root (in the algebraic closure of $\Bbb F_p$) of the 
polynomial $x^p-x-1$. This is a cyclic Galois extension of degree $p$ 
over $\Bbb F_p$. Note that the other roots of $x^p-x-1$ are $\theta +i$
for $i=1,2,\ldots,p-1$. The reason this field extension comes up naturally as
follows. The generating series $F(x)$ of the Bell numbers can be evaluated 
modulo $p$; this means one computes a ``simpler'' series $F_p(x)$
such that $F(x)-F_p(x)$ has all coefficients multiples of $p$, 
where        
 $$
F(x)=\sum_{n=0}^{\infty}B_nx^n=\sum_{n=0}^{\infty}
\frac{x^n}{(1-x)(1-2x)\cdots (1-nx)}
 $$
is the generating function for $B_n$'s. Since Kurepa's hypothesis 
is about the Bell numbers $B_{p-1}$ considered modulo $p$, it makes sense 
to consider $F_p(x)$ rather than $F(x)$. 
By using this idea, D. Barsky and B. Benzaghou 
\cite[Th\'{e}or\`{e}me 3, p. 13]{bb} proved that 
$B_{p-1}\not\equiv 1(\bmod{\,p})$ for any prime $p$.
However, as noticed above,  this proof contains some irreparable calculation 
errors \cite{bb2}.

In view of the previous mentioned formulation of 
Kurepa's hypothesis in terms of Bell numbers  and the equivalence 
$(i) \Leftrightarrow (iii)$ of Theorem 3, it can be of interest to 
determine $B_{n-1}(\bmod{\,n})$ for composite integers $n$. 
A computation in {\tt Mathematica 8} shows that there are 
certain positive integers $n$ such that  $B_{n-1}\equiv 1(\bmod{\,n})$.
All these values of $n$ less than $20000$ are 
$2,4=2^2,16=2^4,28=2^2\cdot 7,46=2\cdot 23,134=2\cdot 67,454=2\cdot 227,
1442=2\cdot 7\cdot 103,1665=3^2\cdot 5\cdot 37$ and $4252=2^2\cdot 1063$.
Notice also that for these values of $n$ the residues $S_{n-1}(\bmod{\,n})$
are respectively as follows: $0,2,6,-6,2,-2,2,568,-476$ and $22$.

We propose the following conjecture.

\begin{conjecture}
There are infinitely many positive integers $n$ such that
  $$
B_{n-1}\equiv 1 \pmod{n}.
 $$
  \end{conjecture}    

Recently,  Z.-W. Sun and D. Zagier \cite[Theorem 1.1]{sz} proved 
that for every positive integer $m$ and any prime $p$ not dividing $m$ we have
  $$
\sum_{k=1}^{p-1}\frac{B_k}{(-m)^k}\equiv (-1)^{m-1}S_{m-1}\pmod{p}.\leqno(12) 
  $$
By using the congruence (12) easily folllows the following result.
  \begin{proposition}
An odd  prime $p$ is a counterexample to  Kurepa's hypothesis
if and only if in the field $\Bbb F_p$ there holds
   $$
\left(\begin{array}{ccccc}
1 & (p-1)^{p-2} & (p-1)^{p-3} & \dots & (p-1)\\
1 & (p-2)^{p-2} & (p-2)^{p-3} & \dots & (p-2)\\
\vdots &\vdots & \vdots & \vdots & \vdots \\
1 & 2^{p-2} & 2^{p-3} & \ddots & 2\\
1 & 1 & 1 & \dots & 1
 \end{array}\right)\left(\begin{array}{c}
B_0\\
B_1\\
B_2\\
\vdots\\
B_{p-2}\end{array}\right)=
\left(\begin{array}{c}
D_0\\
-D_1\\
D_2\\
\vdots\\
-D_{p-2}\end{array}\right),\leqno(13)
  $$
or equivalently,
 $$
\left(\begin{array}{ccccc}
1 & 1 & \dots &  1 & 1\\
(p-1) & (p-2) & \dots & 2  & 1\\
(p-1)^2 & (p-2)^2 & \dots & 2^2  & 1\\
\vdots &\vdots & \vdots & \ddots & \vdots \\
(p-1)^{p-2} & (p-2)^{p-2} & \dots  & 2^{p-2}  & 1
 \end{array}\right)\left(\begin{array}{c}
-D_0\\
D_1\\
-D_2\\
\vdots\\
D_{p-2}\end{array}\right)=
\left(\begin{array}{c}
B_0\\
B_1\\
B_2\\
\vdots\\
B_{p-2}\end{array}\right),\leqno(14)
  $$
\end{proposition}
 \begin{proof}
First observe that by Fermat little theorem 
(cf. \cite[Proof of Corollary 1.2]{sz}), 
   $$
 \sum_{k=1}^{p-1}(p-i)^{p-k}(p-j)^{k-1}
\equiv \sum_{k=1}^{p-1}\left(\frac{p-j}{p-i}\right)^{k-1}
\equiv
\left\{
\begin{array}{cc}
 -1 \pmod{p} &  {\rm if}\,\, i=j;\\
 0 \pmod{p} &  {\rm if}\,\, 1\le i\not=j\le p-1.\\
 \end{array}\right.
 $$
The above congruence shows that for $(p-1)\times (p-1)$  matrices 
$A=\left((p-i)^{p-j}\right)_{1\le i,j\le p-1}$ and 
$B=\left((p-j)^{i-1}\right)_{1\le i,j\le p-1}$
from the left hand sides of (13) and (14), respectively, 
we have $A\cdot B=-I_{p-1}$, where $I_{p-1}$  is the identity 
matrix of order $p-1$. This shows that 
$B=-A^{-1}$ and therefore, (13) yields (14).

In order to prove (13), by using the congruence (12) and Fermat little theorem,
for any odd prime $p$ and each $m=1,2,\ldots, p-1$ we find that 
  \begin{eqnarray*}
\sum_{k=0}^{p-2}\frac{B_k}{(-m)^k}
&\equiv & (-1)^{m-1}S_{m-1}+B_0-\frac{B_{p-1}}{(-m)^{p-1}}\pmod{p}\\
(15)\qquad\qquad\qquad\qquad\qquad &\equiv& (-1)^{m-1}S_{m-1}+B_0-B_{p-1}
\pmod{p}.\qquad\qquad\qquad\qquad
   \end{eqnarray*}
  As noticed above, an odd  prime $p$ is a counterexample to  Kurepa's 
hypothesis if and only if $B_{p-1}\equiv 1(\bmod{\,p})$, which substituting
together with $B_0=1$  into (15) gives 
  $$
\sum_{k=0}^{p-2}\frac{B_k}{(-m)^k}\equiv  (-1)^{m-1}S_{m-1}\pmod{p},
\,\, m=1,2,\ldots,p-1. \leqno(16)
  $$
Since by Fermat little theorem, 
$1/(-m)^k\equiv 1/(p-m)^k\equiv (p-m)^{p-1-k}(\bmod{\,p})$
for all  pairs $(m,k)$ with $1\le m\le p-1$ and  $0\le k\le p-2$,
the congruences (16) can be written as 
   $$
\sum_{k=0}^{p-2}(p-m)^{p-1-k}B_k\equiv  (-1)^{m-1}S_{m-1}\pmod{p},
\,\, m=1,2,\ldots,p-1. \leqno(17)
  $$
Finally, observe that the set of $(p-1)$ congruences 
modulo $p$ given by (17) is equivalent with the matrix equality 
(13) in the field $\Bbb F_p$.
  \end{proof}
\begin{remark} Notice that $(p-1)\times (p-1)$  matrix 
$A=\left((p-i)^{p-j}\right)_{1\le i,j\le p-1}$ (in the field $\Bbb F_p$) 
on the left hand side of (13) is a 
Vandermonde-type matrix. Namely, interchanging 
$j$th column and $(p+1-j)$th column of $A$ for each $j=2,3,\ldots,(p-1)/2$
($(p+1)$th column of $A$ remains fixed), the matrix $A$ becomes the 
Vandermonde matrix $A'=\left((p-i)^j\right)_{1\le i,j+1\le p-1}$.
Hence, 
  \begin{eqnarray*}
\det (A)&=&(-1)^{(p-3)/2}\det (A')=
(-1)^{(p-3)/2}\prod_{1\le i<j\le p-1}((p-j)-(p-i))\\
&=&(-1)^{(p-3)/2}\prod_{1\le i<j\le p-1}(i-j),
 \end{eqnarray*}
whence by using Wilson theorem it can be easily show that
 $$
\det (A)\equiv (-1)^{(p^2-1)/8}\left(\frac{p-1}{2}\right)!\pmod{p}.
\leqno(18)
  $$
In particular, if $p\equiv 3(\bmod{\,4})$, then by a congruence 
of Mordell \cite{mo}, (18) implies that 
 $$
\det (A)\equiv (-1)^{(p^2-1)/8+(h(-p)+1)/2}(\bmod{\,p}),
 $$
where $h(-p)$ is the class number of the imaginary quadratic field 
$\Bbb Q(\sqrt{-p})$. 
Moreover, if $p\equiv 1(\bmod{\,4})$, then  applying  Wilson theorem (18) 
yields 
$$
\left(\det (A)\right)^2\equiv -1(\bmod{\,p}).
 $$
   \end{remark}

\section{Proof of Proposition 2}
   In order to prove Proposition 2, we will need the following lemma.
 \begin{lemma}
For any integer $n\ge 3$ let 
$D_n$ be the determinant of order $n$ defined as
     $$
D_n=\left|\begin{array}{rrrrrrrrrcc}
1 & 1 & 1 & 1 & 1 & 1 &\dots    &1 &   1 & 1 \\
1 & 0 & 1 & 1 & 1 & 1 &\dots  &1 &   1 & 1  \\
0 & 1 & 1 & 1 & 1 & 1 &\dots  &1 &   1 & 1 \\
0 & 0 & 1 & 0 & 1 & 1 &\dots  &1 &   1 & 1 \\
0 & 0 & 0 & 1 & 1 & 1 & \dots &1    & 1 & 1 \\
0 & 0 & 0 & 0 & 1 & 0 & \dots &1   & 1 & 1 \\
0 & 0 & 0 & 0 & 0 & 1 & \dots &1    & 1 &1 \\
\vdots & \vdots &\vdots &\vdots&\vdots  &\vdots &\vdots &\vdots 
&\ddots &\vdots \\
0 & 0 & 0 & 0 & 0 & 0 & \dots & 0  &  1 & (1-(-1)^n)/2  
\end{array}\right|.\leqno(19)
  $$  
$($For example, $D_3=\left|\begin{array}{rrr}
1 & 1 & 1\\
1 & 0 & 1\\
0 & 1 & 1\\
 \end{array}\right|$, 
$D_4=\left|\begin{array}{rrrr}
1 & 1 & 1 &1\\
1 & 0 & 1 &1\\
0 & 1 & 1 & 1\\
0 & 0 & 1 & 0
 \end{array}\right|$,
$D_5=\left|\begin{array}{rrrrr}
1 & 1 & 1 &1 &1\\
1 & 0 & 1 &1 &1\\
0 & 1 & 1 & 1 &1\\
0 & 0 & 1 & 0 &1\\
0 & 0 & 0 & 1 &1
 \end{array}\right|$
  $)$.
Then $D_3=-1$ and 
  $$
D_{2n}=D_{2n+1}=(-1)^n \leqno(20)
   $$
for all $n\ge 2$.
 \end{lemma}
\begin{proof}
If $n\ge 2$, then subtracting the $2n$th column from the $(2n-1)$th column of 
$D_{2n}$, and thereafter expanding the determinant along the $2n$th 
(last) column,   we find that
    \begin{eqnarray*}
D_{2n}&=&\left|\begin{array}{rrrrrrrrrrr}
1 & 1 & 1 & 1 & 1 & 1 &\dots    &1 &  1 & 0 \\
1 & 0 & 1 & 1 & 1 & 1 &\dots  &1 &  1 & 0  \\
0 & 1 & 1 & 1 & 1 & 1 &\dots  &1 &  1 & 0 \\
0 & 0 & 1 & 0 & 1 & 1 &\dots  &1 &  1 & 0 \\
0 & 0 & 0 & 1 & 1 & 1 & \dots &1    & 1 & 0 \\
0 & 0 & 0 & 0 & 1 & 0 & \dots &1   & 1 & 0 \\
\vdots & \vdots &\vdots &\vdots&\vdots  &\vdots &\vdots &\vdots&\ddots 
&\vdots \\
0 & 0 & 0 & 0 & 0 & 0 & \dots & 0  & 1 & -1  
\end{array}\right|\\
(21)\qquad\qquad &=&-
\left|\begin{array}{rrrrrrrrrr}
1 & 1 & 1 & 1 & 1 & 1 &\dots    &1 &  1  \\
1 & 0 & 1 & 1 & 1 & 1 &\dots  &1 &  1   \\
0 & 1 & 1 & 1 & 1 & 1 &\dots  &1 &  1  \\
0 & 0 & 1 & 0 & 1 & 1 &\dots  &1 &  1  \\
0 & 0 & 0 & 1 & 1 & 1 & \dots &1    & 1  \\
0 & 0 & 0 & 0 & 1 & 0 & \dots &1   & 1  \\
\vdots & \vdots &\vdots &\vdots&\vdots  &\vdots &\vdots &\ddots&\vdots 
\\
0 & 0 & 0 & 0 & 0 & 0 & \dots & 1  & 1   
\end{array}\right|=-D_{2n-1}.\qquad\qquad\qquad\qquad
  \end{eqnarray*}

Similarly, if $n\ge 3$, then subtracting the $(2n-1)$th column from the $(2n-2)$th 
column of $D_{2n-1}$, then  expanding the determinant along the 
$(2n-1)$th (last) row, and 
thereafter expanding the determinant along the 
$(2n-2)$th (last) column,  we find that
          \begin{eqnarray*}
D_{2n-1}&=&\left|\begin{array}{rrrrrrrrrrr}
1 & 1 & 1 & 1 & 1 & 1 &\dots    &1 &  1 & 0 \\
1 & 0 & 1 & 1 & 1 & 1 &\dots  &1 &  1 & 0  \\
0 & 1 & 1 & 1 & 1 & 1 &\dots  &1 &  1 & 0 \\
0 & 0 & 1 & 0 & 1 & 1 &\dots  &1 &  1 & 0 \\
0 & 0 & 0 & 1 & 1 & 1 & \dots &1    & 1 & 0 \\
0 & 0 & 0 & 0 & 1 & 0 & \dots &1   & 1 & 0 \\
\vdots & \vdots &\vdots &\vdots&\vdots  &\vdots &\vdots &\ddots&\vdots 
&\vdots \\
0 & 0 & 0 & 0 & 0 & 0 & \dots & 1  & 0 & 1  \\
0 & 0 & 0 & 0 & 0 & 0 & \dots & 0  & 1 & 0  
\end{array}\right|\\
(22)\qquad\qquad\qquad &=&-
\left|\begin{array}{rrrrrrrrrrr}
1 & 1 & 1 & 1 & 1 & 1 &\dots    &1 &  1 & 0 \\
1 & 0 & 1 & 1 & 1 & 1 &\dots  &1 &  1 & 0  \\
0 & 1 & 1 & 1 & 1 & 1 &\dots  &1 &  1 & 0 \\
0 & 0 & 1 & 0 & 1 & 1 &\dots  &1 &  1 & 0 \\
0 & 0 & 0 & 1 & 1 & 1 & \dots &1    & 1 & 0 \\
0 & 0 & 0 & 0 & 1 & 0 & \dots &1   & 1 & 0 \\
\vdots & \vdots &\vdots &\vdots&\vdots  &\vdots &\vdots &\ddots&\vdots 
&\vdots \\
0 & 0 & 0 & 0 & 0 & 0 & \dots & 1  & 1 & 0  \\
0 & 0 & 0 & 0 & 0 & 0 & \dots & 0  & 1 & 1  
\end{array}\right|\qquad\qquad\qquad\qquad\\
&=& -\left|\begin{array}{rrrrrrrrrrr}
1 & 1 & 1 & 1 & 1 & 1 &\dots    &1 &  1 & 1 \\
1 & 0 & 1 & 1 & 1 & 1 &\dots  &1 &  1 & 1  \\
0 & 1 & 1 & 1 & 1 & 1 &\dots  &1 &  1 & 1 \\
0 & 0 & 1 & 0 & 1 & 1 &\dots  &1 &  1 & 1 \\
0 & 0 & 0 & 1 & 1 & 1 & \dots &1    & 1 & 1 \\
0 & 0 & 0 & 0 & 1 & 0 & \dots &1   & 1 & 1 \\
\vdots & \vdots &\vdots &\vdots&\vdots  &\vdots &\vdots &\ddots&\vdots 
&\vdots \\
0 & 0 & 0 & 0 & 0 & 0 & \dots & 1  & 0 & 1  \\
0 & 0 & 0 & 0 & 0 & 0 & \dots & 0  & 1 & 1  
\end{array}\right|=-D_{2n-3}.
  \end{eqnarray*}
From (22) and the fact that $D_3=\left|\begin{array}{rrr}
1 & 1 & 1\\
1 & 0 & 1\\
0 & 1 & 1\\
 \end{array}\right|=-1$ it follows that for $n\ge 2$
$D_{2n-1}=(-1)^{n-1}$, which substituting in (21) gives 
$D_{2n}=(-1)^n$. This shows that $D_{2n}=D_{2n+1}=(-1)^n$
for all $n\ge 2$, as desired.
\end{proof}
 \begin{proof}[Proof of Proposition $2$]
By (6) of Definition 2 with $2n$ instead of $n$ we have
 $$
K_{2n}'=\left|\begin{array}{cccccccccccc}
1 & 1& 1 & 1 &  \dots   &1 & 1  & 1 & 1 &1 \\
1 & 1 & 1 & 1 & \dots    &1 & 1 & 1 & 1 &0 \\
1 & 0 & 1 & 1 & \dots  &1 & 1 & 1 & 1 & 0 \\
0 & 1 & 1 & 1 & \dots  &1 & 1 & 1 & 1 & 0\\
0 & 0 & 1 & 0 & \dots  &1 & 1 & 1 & 1 & 0\\
0 & 0 & 0 & 1 & \dots &1   & 1 & 1 & 1 & 0\\
0 & 0 & 0 & 0 & \dots &1  & 1 & 1 & 1 & 0\\
\vdots & \vdots &\vdots &\vdots &\vdots &\vdots &\vdots&\ddots 
&\vdots &\vdots\\
0 & 0 & 0 & 0 & \dots  & 0 &  1 & 0  & 1 & 0\\
0 & 0 & 0 & 0 &  \dots & 0 & 0 & 0 & 1 & 0
\end{array}\right|,
  $$
(the determinant $K_{2n}'$ is of order $2n-4$),
whence after the expansion along the $(2n-4)$th (last) column,
and then expanding this along the $(2n-5)$th (last) row,
for all $n\ge 5$ we obtain
\begin{eqnarray*}
K_{2n}'&=&-\left|\begin{array}{cccccccccccc}
1 & 1& 1 & 1 &  \dots   &1 & 1  & 1 & 1  \\
1 & 0 & 1 & 1 & \dots  &1 & 1 & 1 & 1  \\
0 & 1 & 1 & 1 & \dots  &1 & 1 & 1 & 1 \\
0 & 0 & 1 & 0 & \dots  &1 & 1 & 1 & 1 \\
0 & 0 & 0 & 1 & \dots &1   & 1 & 1 & 1 \\
0 & 0 & 0 & 0 &  \dots &1  & 1 & 1 & 1 \\
\vdots & \vdots &\vdots &\vdots &\vdots &\vdots &\ddots&\vdots 
&\vdots \\
0 & 0 & 0 & 0 & \dots  & 0 &  1 & 0  & 1\\ 
0 & 0 & 0 & 0 & \dots  & 0 &  0 & 0  & 1 
\end{array}\right|\\
&=&
-\left|\begin{array}{cccccccccccc}
1 & 1& 1 & 1 &  \dots   &1 & 1  & 1 & 1  \\
1 & 0 & 1 & 1 & \dots  &1 & 1 & 1 & 1  \\
0 & 1 & 1 & 1 & \dots  &1 & 1 & 1 & 1 \\
0 & 0 & 1 & 0 & \dots  &1 & 1 & 1 & 1 \\
0 & 0 & 0 & 1 & \dots &1   & 1 & 1 & 1 \\
0 & 0 & 0 & 0 & \dots &1  & 1 & 1 & 1 \\
\vdots & \vdots &\vdots &\vdots &\vdots &\vdots& \ddots&\vdots 
&\vdots \\
0 & 0 & 0 & 0 & \dots  & 0 &  1 & 1  & 1\\ 
0 & 0 & 0 & 0 & \dots  & 0 &  0 & 1  & 0 
\end{array}\right|=-D_{2n-6},
  \end{eqnarray*}
where the determinant $D_{2n-6}$ is defined by (19).
Then the above equality and (20) of Lemma 1 yield 
  $$
K_{2n}'=-D_{2(n-3)}=-(-1)^{n-3}=(-1)^n\quad 
{\rm for\,\, all\,\,} n\ge 5.\leqno(23)
   $$
Further, by (6) of Definition 2 with $2n-1$ instead of $n$ we have
  $$
K_{2n-1}'=\left|\begin{array}{cccccccccccc}
1 & 1& 1 & 1 &  \dots   &1 & 1  & 1 & 1 &1 \\
1 & 1 & 1 & 1 & \dots    &1 & 1 & 1 & 1 &0 \\
1 & 0 & 1 & 1 & \dots  &1 & 1 & 1 & 1 & 0 \\
0 & 1 & 1 & 1 & \dots  &1 & 1 & 1 & 1 & 0\\
0 & 0 & 1 & 0 & \dots  &1 & 1 & 1 & 1 & 0\\
0 & 0 & 0 & 1 & \dots &1   & 1 & 1 & 1 & 0\\
0 & 0 & 0 & 0 & \dots &1  & 1 & 1 & 1 & 0\\
\vdots & \vdots &\vdots &\vdots &\vdots &\vdots &\vdots&\ddots 
&\vdots &\vdots\\
0 & 0 & 0 & 0 & \dots & 0 &  1 & 1  & 1 & 0\\
0 & 0 & 0 & 0 &  \dots & 0 & 0 & 0 & 1 & 0
\end{array}\right|,
  $$
(the determinant $K_{2n-1}'$ is of order $2n-5$),
whence after the expansion along the $(2n-5)$th (last) column,
and then expanding this along the $(2n-6)$th (last) row,
for all $n\ge 5$ we obtain
\begin{eqnarray*}
K_{2n-1}'&=&\left|\begin{array}{cccccccccccc}
1 & 1& 1 & 1 &  \dots   &1 & 1  & 1 & 1  \\
1 & 0 & 1 & 1 & \dots  &1 & 1 & 1 & 1  \\
0 & 1 & 1 & 1 & \dots  &1 & 1 & 1 & 1 \\
0 & 0 & 1 & 0 & \dots  &1 & 1 & 1 & 1 \\
0 & 0 & 0 & 1 & \dots &1   & 1 & 1 & 1 \\
0 & 0 & 0 & 0 &  \dots &1  & 1 & 1 & 1 \\
\vdots & \vdots &\vdots &\vdots &\vdots &\vdots &\ddots&\vdots 
&\vdots \\
0 & 0 & 0 & 0 & \dots  & 0 &  1 & 1  & 1\\ 
0 & 0 & 0 & 0 & \dots  & 0 &  0 & 0  & 1 
\end{array}\right|\\
&=&
-\left|\begin{array}{cccccccccccc}
1 & 1& 1 & 1 &  \dots   &1 & 1  & 1 & 1  \\
1 & 0 & 1 & 1 & \dots  &1 & 1 & 1 & 1  \\
0 & 1 & 1 & 1 & \dots  &1 & 1 & 1 & 1 \\
0 & 0 & 1 & 0 & \dots  &1 & 1 & 1 & 1 \\
0 & 0 & 0 & 1 & \dots &1   & 1 & 1 & 1 \\
0 & 0 & 0 & 0 & \dots &1  & 1 & 1 & 1 \\
\vdots & \vdots &\vdots &\vdots &\vdots &\vdots& \ddots&\vdots 
&\vdots \\
0 & 0 & 0 & 0 & \dots  & 0 &  1 & 0  & 1\\ 
0 & 0 & 0 & 0 & \dots  & 0 &  0 & 1  & 1 
\end{array}\right|=D_{2n-7},
  \end{eqnarray*}
where the determinant $D_{2n-7}$ is defined by (19).
Then the above equality and (20) of Lemma 1 yield 
  $$
K_{2n-1}'=D_{2(n-4)+1}=(-1)^{n-4}=(-1)^{n}\quad 
{\rm for\,\, all\,\,} n\ge 5.\leqno(24)
   $$
Finally, the equalities (23) and (24) and the fact that 
$K_7'=K_8'=1$ yield the assertion of Proposition 2.
 \end{proof}

 \end{document}